\documentclass[reqno]{amsart}
\usepackage[T1]{fontenc}
\usepackage{dsfont}
\usepackage{mathrsfs}
\usepackage[colorlinks]{hyperref}
\usepackage{xcolor}
\usepackage[a4paper,asymmetric]{geometry}
\usepackage{mathscinet}
\usepackage{latexsym}
\usepackage{amsthm}
\usepackage{amssymb}
\usepackage{amsfonts}
\usepackage{amsmath}
\usepackage{longtable}
\usepackage{graphicx}
\usepackage{multirow}
\usepackage{multicol}

\newtheorem{theorem}{Theorem}[section]
\newtheorem{thm}[theorem]{Theorem}

\newtheorem{lem}[theorem]{Lemma}
\newtheorem{remark}[theorem]{Remark}
\newtheorem{proposition}[theorem]{Proposition}
\newtheorem{prop}[theorem]{Proposition}
\newtheorem{corollary}[theorem]{Corollary}

\theoremstyle{definition}
\newtheorem{definition}[theorem]{Definition}
\newtheorem{defn}[theorem]{Definition}
\newtheorem{ex}[theorem]{Example}
\newtheorem{ques}[theorem]{Question}
 
\theoremstyle{remark}
\numberwithin{equation}{section}
 \DeclareMathOperator{\RE}{Re}
 \DeclareMathOperator{\IM}{Im}
 
 \DeclareMathAlphabet{\mathpzc}{OT1}{pzc}{m}{it}
 \DeclareMathAlphabet{\mathsfsl}{OT1}{cmss}{m}{sl}

  \newcommand{\DR}{\mathbb{D}}
  \newcommand{\FH}{\mathfrak{H}}

\newcommand{\dif}{\mathrm{d}}

\newcommand{\abs}[1]{\left\vert#1\right\vert}
\newcommand{\set}[1]{\left\{#1\right\}}

\newcommand{\norm}[1]{\left\Vert#1\right\Vert}

\newcommand{\E}{\mathbb{E}}

\newcommand{\R}{\mathbb{R}}

 \newcommand{\tensor}[1]{\mathsf{#1}}

 \newcommand{\Rnum}{\mathbb{R}}
 \newcommand{\Cnum}{\mathbb{C}}
 
 \newcommand{\Nnum}{\mathbb{N}}

 \newcommand{\innp}[1]{\langle {#1}\rangle}
\newcommand{\Be}{\begin{equation}}
\newcommand{\Ee}{\end{equation}}
\newcommand{\Bs}{\begin{split}}
\newcommand{\Es}{\end{split}}
\newcommand{\Bes}{\begin{equation*}}
\newcommand{\Ees}{\end{equation*}}
\newcommand{\BT}{\begin{thm}}
\newcommand{\ET}{\end{thm}}
\newcommand{\Bp}{\begin{proof}}
\newcommand{\Ep}{\end{proof}}
\newcommand{\BL}{\begin{lem}}
\newcommand{\EL}{\end{lem}}
\newcommand{\BP}{\begin{proposition}}
\newcommand{\EP}{\end{proposition}}
\newcommand{\BC}{\begin{corollary}}
\newcommand{\EC}{\end{corollary}}
\newcommand{\BR}{\begin{remark}}
\newcommand{\ER}{\end{remark}}
\newcommand{\BD}{\begin{defn}}
\newcommand{\ED}{\end{defn}}
\newcommand{\BI}{\begin{itemize}}
\newcommand{\EI}{\end{itemize}}

\newcommand{\mi}{{\rm i}}

\allowdisplaybreaks

\begin{document}

\title[Complex Multiple Wiener-It\^o Chaos]{Complex Wiener-It\^o Chaos Decomposition Revisited }
\author[Y. Chen]{Yong CHEN}
 \address{College of Mathematics and Information Science, Jiangxi Normal University, Nanchang, 330022, Jiangxi, China}
\email{zhishi@pku.org.cn}
\author[Y. Liu]{Yong LIU}
 \address{LMAM, School of Mathematical Sciences, Peking University, Beijing, 100871, China }
\email{liuyong@math.pku.edu.cn}
\begin{abstract}  In this article, some properties of complex Wiener-It\^o multiple integrals and complex Ornstein-Uhlenbeck operators and semigroups are obtained. Those include Stroock's formula,  Hu-Meyer formula, Clark-Ocone formula and the hypercontractivity of complex Ornstein-Uhlenbeck semigroups. As an application, several expansions of the fourth moments of complex Wiener-It\^o  multiple integrals are given.\\

{\bf Keywords:} Complex Hermite polynomials, Complex Gaussian isonormal processes,  Complex Wiener-It\^o Multiple Integrals, Complex Ornstein-Uhlenbeck operators and semigroups\\

{\bf MSC 2000:} 60H07; 60F25; 62M09.
\end{abstract}
\maketitle
\tableofcontents


\section{Introduction}
The chaos decomposition is an elegant and profound theory in stochastic analysis.  As a powerful analytic approach,  it has been applied to many fields,  for example,  mathematical physics, statistics and mathematical finance (see \cite{DiOP}\cite{HP-A}\cite{hyz 17}\cite{hy 97}\cite{malliavin}\cite{janson}\cite{np}\cite{Nualart} and references therein).

Wiener published the seminal article of this theory \textit{The Homogenous Chaos} \cite{wie 38} in 1938.  His original idea is that a singular  or noisy single process can be decomposed  into  a series of the usual monomials of some \textit{simply} random variables called polynomial chaos.  Moreover, these polynomial chaos is represented by multiple Wiener integral named by It\^o \cite{ito 51}. The main methods are the theory of  generalized harmonic analysis developed by himself and a generalized space-time Birkhoff's individual ergodic theorem.  Actually,  \textit{chaos}, in Wiener's mind, implies more meanings including randomness, singularity even disorder. 

Cameron and Martin, in 1947 \cite{cm 47}, constructed a complete orthogonal basis on $L^2(C_0[0,1], P^B)$,
\begin{equation}\label{cm}
 \Phi_{m,p}=H_m\big(\int_0^1 \alpha_p(s)\dif B(s)\big),
\end{equation}
where, $(C_0[0,1], P^B)$ is the classic Wiener space on $[0,1]$, $H_m$ is the Hermite polynomial of $m$ degree  which is defined as $H_m(x)=(-1)^m  e^{x^2/2}\frac{\dif^m}{\dif x^m}e^{-x^2/2},\,m=0,1,\dots $, $\alpha_p$ is an orthogonal basis on $L^2(0,1)$ and $\int_0^1 \alpha_p(s)\dif B(s)$ is the It\^o stochastic integral with respect to Wiener process on $[0,1]$.

In 1951 \cite{ito 51},  It\^o introduced a stochastic multiple integral and modified Wiener's one \cite{wie 38}. It\^o showed the orthogonality of his multiple integrals of different degrees, especially the Hermite polynomial in Eq. (\ref{cm}) can be expressed by this  multiple integral as follows \cite{ito 51} \cite{Nualart}:
$$
 H_m\big(\int_0^1 \alpha_p(s)\dif B(s)\big)= \int_{[0,1]^m} \alpha_p(t_1)\cdots\alpha_p(t_m)\dif B(t_1)\cdots \dif B(t_m).
$$
Moreover, It\^o showed a relation of the above multiple integrals and the iterated It\^o integrals. 
 Therefore, he gave a decomposition of orthogonal direct sum of $L^2$ functionals of Gaussian system, which is the standard form in the present textbooks and  literature.

Being motivated in part by constructive quantum field theory, in 1956 \cite{segal}, Segal  tried to develop a theory of integrations and harmonic analysis on Hilbert spaces. Although Segel's theory is only based on finite additive (cylindrical) measure, he indicated the construction equivalence  between the square integrable functions on a real Hilbert space $\FH_{\R}$ and the algebra of symmetric tensors over $\FH$, the complexification of $\FH_{\R}$. This idea is a key point of the definition of  multiple stochastic integrals on abstract spaces. In  1965 \cite{Gross}, Gross established the abstract Wiener space, which provides a proper measure theory and effective analytic framework on infinite dimensional Gaussian spaces. 

In 1987, in a short note \cite{str 87}, Stroock gave a useful approach to compute the Wiener-It\^o chaos decomposition by Malliavin' s derivative (or Gross-Sobolev derivative) and divergence operator, which is named as Stroock's formula \cite{Nualart} \cite{np}. This formula sets up a deep link between chaos decomposition and Malliavin calculus.  

One can summarize all of the above theories in the following theorem (\cite{ito 51}\cite{Nualart}\cite{str 87}\cite{np}\cite{hyz 17}).

\begin{thm}
  Let $X=\{X(h), h\in \FH_\mathbb{R}\}$ be an isonormal Gaussian process over $\FH_\mathbb{R}$.    \begin{itemize}
    \item[\textup{(i)}]  The linear space generated by the class $\{H_n(X(h)): n\ge 0, h\in \FH_\mathbb{R}, \|h\|_{\FH_\mathbb{R}}=1\}$ is dense in $L^q(\Omega, \sigma(X), P)$, $q\in [1,\infty)$.      
    
    \item[\textup{(ii)}] (Wiener-It\^o chaos decomposition) \ Denote $\mathcal{H}_n$ by the closed linear subspace of $L^2(\Omega, \sigma(X), P)$ generated by $\{H_n(X(h)): h\in \FH_\mathbb{R}, \|h\|_{\FH_\mathbb{R}}=1\}$. One has that 
    $$
     L^2(\Omega, \sigma(X), P)=\bigoplus_{n=0}^\infty\mathcal{H}_n.
    $$
    This means that every $F\in L^2(\Omega, \sigma(X), P)$ admits a unique expansion of the type $F=E(F)+\sum_{n=1}^\infty F_n$, where $F_n\in \mathcal{H}_n$ and the series converges in $L^2(\Omega, \sigma(X), P)$. 
    \item[\textup{(iii)}] (Stroock's formula) Denote $\FH_\mathbb{R}^{\odot n}$ by the $n$ times symmetric tensor space of $\FH_\mathbb{R}$ equipped with modified norm $\sqrt{n!}\|\cdot\|_{\FH_\mathbb{R}^{\otimes n}}$. The $n$-th multiple integral of $f\in \FH_\mathbb{R}^{\odot n}$ with respect to $X$ defined by $\mathcal{I}_n(f)=\delta^p(f)$, where $\delta^n$ is the divergence operator of order $n$. Then, every $F\in L^2(\Omega, \sigma(X), P)$ can be expanded as 
    $$
     F=E(F)+\sum_{n=1}^\infty \mathcal{I}_n(f_n),
    $$
   for some unique collection of kernels $f_n\in \FH_\mathbb{R}^{\odot n}$, $n\ge 1$. Moreover, if $F\in \mathbb{D}^{n,2}$ for some $n\ge 1$, then $f_p=\frac{1}{p!}E(D^pF)$ for all $p\le n$, where $D^n$ is the $n$-th Malliavin derivative operator and $\mathbb{D}^{n,2}$ is its domain in $L^2(\Omega)$.
   \end{itemize}
\end{thm} 

This theorem not only clearly characterizes the square integrable functional on Wiener space through Hermite polynomials, Wiener-It\^o multiple integrals and Malliavin's derivative  (or Gross-Sobolev derivative), but also reveals their intrinsic connection.  Moreover, these terms relate the exponential martingale, linear stochastic differential equation, Ornstein-Uhlenbeck operator, Ornstein-Uhlenbeck semigroup and the Wick product interestingly. 

On the one hand, it is well known that the solution of the following linear stochastic differential equation is an exponential martingale, which can be expanded by a serial of Hermite polynomials, 
$$
\dif x(t)=\phi(t)x(t)\dif B(t), \ \ x(0)=1.
$$
On the other hand, solving this equation by Picard's iteration, the $n$-fold iterated Wiener-It\^o's integrals is represented naturally by  $n$-degree Hermite polynomial in the expansion of the exponential martingales  \cite{gong}\cite{ks}\cite{ry}.   Moreover, this relation is also interpreted by Wick exponent and Wick product \cite{janson}.

One defines the gradient operator $D$ on $L^2(\mathbf{ R}^d, \dif \nu)$, where $\dif \nu$ is the standard Gaussian measure on $\mathbf{ R}^d$. The divergence operator $\delta$ is the adjoint operator with respect to the inner product $\int_{\mathbf{ R}^d}\langle\cdot, \cdot\rangle\dif \nu$ by using integration by parts.  OU operator $\Delta_{OU}\stackrel{\rm def}{=}-\delta D$, is the infinitesimal generator of OU process (or OU semigroup) possessing the invariable distribution $\dif \nu$.  Moreover, the Hermite polynomials of several variables are the eigenfunctions  of $\Delta_{OU}$ and a complete orthonormal basis on $L^2(\mathbf{ R}^d, \dif \nu)$.

One can lift this picture to infinite dimensional Wiener spaces, and define an infinite dimensional OU operator and OU semigroup, which play the same role as the Laplacian operator $\Delta$ and the heat semigroup in Fourier analysis on $\mathbf{ R}^d$. Using the spectral property of OU semigroup essentially associated with chaos decomposition, one obtains the multiplier theorem and then establishes the theory of Sobolev spaces on Wiener spaces. Here, the hypercontractivity property of OU semigroup and Meyer's inequality are two key tools \cite{Nualart}\cite{np}\cite{hyz 17}.     

These real analytic theories on Wiener spaces have a wide range of important application in the fields of probability theory, stochastic (partial) differential equation, mathematical finance and mathematical physics etc. 

In fact,  in 1953 \cite{ito},  It\^o gave the theory of complex multiple integrals. His main idea was using complex Hermite polynomials to define multiple integrals and then obtaining the chaos decomposition of square integrable functions on complex Gaussian system. In recent years, there are more and more interesting applications of complex chaos decomposition and its related topics (see \cite{af}\cite{cgxx}\cite{chw 17}\cite{cgg}\cite{deck1}\cite{deck2}\cite{hin 17}\cite{is}\cite{kallenberg}).  However,  comparing with the theory of real chaos decomposition,  the corresponding  complex theory is far away to be completed.   

This paper is arranged as follows. In section 2,  after  some preliminaries on complex multiple Wiener-It\^o integrals in subsection 2.1, one obtains complex Stroock's formula in subsection 2.2. The product formula of complex Weiner-It\^o integrals is introduce in subsection {\color{blue}2.3}.  Based on this formula, Wick product, Hu-Meyer formula and the relation between real and complex Weiner-It\^o integrals are given in subsection 2.3, 2.4 and 2.5 respectively.   In section 3,  by means of complex Malliavin divergence operator and Ornstein-Uhlenbeck  operator, Clark-Ocone formula, the hypercontracivity of complex Ornstein-Uhlenbeck semigroups and the expansion of the fourth moment of complex Weiner-It\^o integrals in terms of  Malliavin derivatives and contraction of their kernel are obtained in subsection 3.1, 3.2 and 3.3 respectively.  Finally, section 4 is an appendix, some formulae and facts of complex Hermite polynomials are collected for the reader's convenience.


In the rest of this note, all functions and random variables are complex valued and  Hilbert spaces are  complex inner product spaces in general unless stated precisely.


\section{Complex Wiener-It\^o multiple integrals and chaos decomposition}
In this section, the theory of  complex Wiener-It\^o chaos decomposition is rewriten in term of Gaussian isonormal processes and a version of Stroock's formula is shown.


\subsection{Complex Wiener-It\^o multiple integrals}

The starting point is complex Hermite polynomials.  A simple method to define complex Hermite polynomials is by means of the generating function.
\begin{defn}\it 
Let $z=x + \mi y$ with $x,y\in \Rnum$. Complex Hermite polynomials $J_{m,n}(z,\rho)$ are given by 
  \begin{equation}\label{gene}
    \exp\set{\lambda \bar{z} + \bar{\lambda}z-\rho |\lambda|^2}=\sum_{m=0}^{\infty}\sum_{n=0}^{\infty}\frac{\bar{\lambda}^m\lambda^n}{m!n!}J_{m,n}(z,\rho),\qquad \lambda\in \Cnum.
  \end{equation}
When $\rho=2$, one will often write $J_{m,n}(z)$ rather than $J_{m,n}(z,\rho)$.
\end{defn}

\begin{remark}
As far as the authors known,  
both the formula (\ref{gene}) and the name of complex Hermite polynomials  were introduced by It\^o  firstly in 1953 \cite{ito}. 
\end{remark}

It is well known that one can express $J_{m,n}(z,\rho) $ as follows:
\begin{equation}\label{jal}
  J_{m,n}(z,\,\rho)=\sum_{r=0}^{m\wedge n}(-1)^r r!{m \choose r}{n \choose r}z^{m-r}\bar{z}^{n-r}\rho^{r}.
\end{equation} Some well known properties of these polynomials are listed in Appendix.

Next, one need the definition of complex Gaussian isonormal processes.
Suppose that $\FH$ is a complex separable Hilbert space.  According to \cite{cl2,janson}, a complex Gaussian isonormal process $\set{Z(h):\,h\in \FH}$ is  a centered symmetric complex Gaussian family in $L^2(\Omega)$ such that
\begin{align*}
   E[ Z({h})^2]=0,\quad E[Z({g})\overline{Z({h})}]=\innp{{g}, {h}}_{\FH},\quad \forall {g},{h}\in \FH.
\end{align*} Especially, when $\FH=L^2(T,\mathcal{B},\nu)$ with $\nu$ non-atomic, one  says $\set{Z(A):\, A\in \mathcal{B},\, |\nu(A)|<\infty}$ a white noise with intensity $\nu$, where $Z(A)$ is understood as $Z(\mathrm{1}_{A})$.
   
Third, one need the definition of Wiener-It\^{o} chaos.  For each $m,n\ge 0$, one writes $\mathscr{H}_{m,n}(Z)$ to indicate the closed linear subspace of $L^2(\Omega)$ generated by the random variables of the type
   $$\set{J_{m,n}(Z( {h})):\, {h}\in \mathfrak{H},\,\norm{ {h}}_{\mathfrak{H} }=\sqrt2}$$ where $J_{m,n}(z)$ is the complex Hermite polynomial. The space $\mathscr{H}_{m,n}(Z)$ is called the Wiener-It\^{o} chaos of degree of $(m,n)$ of $Z$ (or say: $(m,n)$-th Wiener-It\^{o} chaos of $Z$).

    Take a complete orthonormal system $\set{ {e}_k}$ in $\mathfrak{H} $. For two sequences $\mathbf{m}=\set{m_k}_{k=1}^{\infty},\,\mathbf{n}=\set{n_k}_{k=1}^{\infty}$ of nonnegative integrals with finite sum, define a Fourier-Hermite polynomial
  \begin{equation}\label{fourier}
    \mathbf{J}_{\mathbf{m},\mathbf{n}}:=\prod_{k} \frac{1}{\sqrt{2^{m_k+n_k}m_k!n_k!}}J_{m_k,n_k}(\sqrt2 Z({e}_k)).
  \end{equation}

The following proposition presents the basis of $(m,n)$-th Wiener-It\^{o} chaos $\mathscr{H}_{m,n}(Z)$ . For its proof, please refer to Theorem 13.1 and Theorem 13.2  in \cite{ito}.
\begin{prop}\label{ppp1}
  Define $\mathbf{m},\,\mathbf{n}$ and $\set{{e}_k}$ be as above.  For any $m,n\ge0$ the random variables
   \begin{equation}\label{eq4}
      \set{\mathbf{J}_{\mathbf{m},\mathbf{n}}:\abs{\mathbf{m}} =m,\,\abs{\mathbf{n}}=n }
   \end{equation}
   form a complete orthonormal system in $\mathscr{H}_{m,n}(Z)$.
\end{prop}

Fourth, one has that the linear mapping
\begin{equation}\label{mulint}
  {I}_{m,n}\big((\hat{\otimes}_{k=1}^{\infty} {e}_k^{\otimes m_k})\otimes (\hat{\otimes}_{k=1}^{\infty}\bar{ {e}}_k^{\otimes n_k})\big)=\sqrt{\mathbf{m}!\mathbf{n}!}\mathbf{J}_{\mathbf{m},\mathbf{n}}
\end{equation} 
provides an isometry from the tensor product $\FH^{\odot m}\otimes \FH^{\odot n}$, equipped with the norm $\sqrt{m!n!}\norm{\cdot}_{\FH^{\otimes (m+n)}}$, onto the $(m,n)$-th Wiener-It\^{o} chaos $\mathscr{H}_{m,n}(Z)$. Here, $e_i\hat{\otimes} e_j$ means that the symmetrizing tensor product of $e_i$ and $e_j$ and $\FH^{\odot n}$ is the $n$ times symmetric tensor  product of $\FH$. Then one can define complex multiple Wiener-It\^{o} integrals as follows.
\begin{defn}\label{imn}\it 
For any $f\in \FH^{\odot m}\otimes \FH^{\odot n}$, one call $ {I}_{m,n}(f)$ the complex multiple Wiener-It\^{o} integral of $f$ with respect to $Z$, since when $\FH=L^2(T,\mathcal{B},\nu)$ with $\nu$ non-atomic, this linear mapping $I_{m,n}:\, \FH^{\odot m}\otimes \FH^{\odot n}\to L^2(\Omega)$ coincides with the classical 
complex multiple Wiener-It\^{o} integral \cite{cl2, ito}.

For any $f\in \mathfrak{H}^{\otimes (m+n)}$ given by $$f= \sum_{j_1,\dots,j_{m+n}}a_{j_1,\dots,j_{m+n}}e_{j_1}\otimes\dots e_{j_m}\otimes \bar{e}_{j_{m+1}}\otimes\dots\otimes \bar{e}_{j_{m+n}},$$
define $I_{m,n}(f)=I_{m,n}(\tilde{f})$ where $\tilde{f}$ is the symmetrization of $f$ in the sense of It\^{o} \cite{ito}:
\begin{equation}\label{ito-sense}
   \tilde{f}=\frac{1}{m!n!}\sum_{\pi}\sum_{\sigma}\sum_{j_1,\dots,j_{m+n}}a_{j_1,\dots,j_{m+n}}e_{j_{\pi(1)}}\otimes\dots\otimes e_{j_{\pi(m)}}\otimes \bar{e}_{j_{\sigma(1)}}\otimes\dots\otimes \bar{e}_{j_{\sigma(n)}},
\end{equation}
where $\pi$ and $\sigma$ run over all permutations of $(1,\dots,m)$ and $(m+1,\dots, m+n)$ respectively.
Especially, when $\FH=L^2(T,\mathcal{B},\nu)$ with $\nu$ non-atomic, one has that
\begin{align*}
  \tilde{f}(t_1,\dots,t_{m+n})=\frac{1}{m!n!}\sum_{\pi}\sum_{\sigma}f(t_{\pi(1)},\dots,t_{\pi(m)},t_{\sigma(1)},\dots,t_{\sigma(n)}).
\end{align*}
\end{defn}

\begin{ex} 
 Similar to the real case, the complex multiple integral connects naturally with an exponential martingale and a linear stochastic differential equation. Using the complex Hermite polynomials, one can express the exponential martingale of a conformal continuous local martingale as series.
 Suppose $x$ is a conformal local martingale and $x_0=0$ and $\lambda\in \Cnum$. Let $y(\lambda)$ be its exponential martingale:
 \begin{align*}
 y(\lambda)=\exp\set{\bar{\lambda} x+\lambda \bar{x}-\abs{\lambda}^2\innp{x,\,\bar{x}}}.
 \end{align*}
 Hence, one can express it as follows, 
 \begin{align*}
  y(\lambda)=\sum_{m=0}^{\infty}\sum_{n=0}^{\infty}\frac{\bar{\lambda}^m\lambda^n}{m!n!}J_{m,n}(x,\innp{x,\,\bar{x}}).
 \end{align*}
 
 On the other hand, $y(\lambda)$ is the unique solution to the linear stochastic differential equation
 \begin{align*}
 \dif y=y (\bar{\lambda} \dif x+\lambda \dif \bar{x}),\qquad y_0=1.
 \end{align*}
 By Picard's iteration, one has that 
 \begin{align*}
 y_t(\lambda)=\sum_{m=0}^{\infty}\sum_{n=0}^{\infty} \bar{\lambda}^m\lambda^n \sum  \int_0^t\int _0^{t_{m+n}}\cdots\int_0^{t_2}\,\dif C_{t_1}\dif C_{t_2}\cdots \dif C_{t_{m+n}},
 \end{align*}  where $0<t_1<t_2<\cdots <t_{m+n}<t$, $C_t=x_t$ or $C_t=\bar{x}_t$, and the sum is over all $n$-combinations of $\set{1,2,\dots,m+n}$ such that $C_t=\bar{x}_t$ .
Comparing coefficients yields that 
\begin{align*}
J_{m,n}(x_t,\innp{x,\,\bar{x}}_t)&=m!n! \sum  \int_0^t\int _0^{t_{m+n}}\cdots\int_0^{t_2}\,\dif C_{t_1}\dif C_{t_2}\cdots \dif C_{t_{m+n}}.
\end{align*}
Especially,  when $x$ is a complex Brownian motion, the right hand side of the above equality is equal to the complex multiple Wiener-It\^o integral: \cite{cl}
\begin{align*}
\int_0^t\int _0^{t}\cdots\int_0^{t}\,\dif x_{t_1}\cdots \dif x_{t_m}\dif \bar{x}_{t_{m+1}}\cdots \dif \bar{x}_{t_{m+n}}.
\end{align*}
\end{ex}


\subsection{Complex Wiener-It\^o chaos decomposition and Stroock's formula}

Now one can give the theory of complex Wiener-It\^{o} chaos decomposition \cite{cl2, ito}.
\begin{thm}\label{chaos}
\begin{itemize}
    \item[\textup{(i)}]The linear space generated by the class
    \begin{equation}\label{jmn2}
       \set{J_{m,n}(Z({h})):\, m,n\ge 0,  {h}\in \mathfrak{H} ,\norm{ {h}}_{\mathfrak{H}}=\sqrt2}
    \end{equation}
   is dense in $L^q(\Omega,\sigma(Z),P)$ for every $q\in [1,\infty)$.
    \item[\textup{(ii)}] One has that $ L^2(\Omega,\sigma(Z),P)=\bigoplus^{\infty}_{m=0}\bigoplus^{\infty}_{n=0}\mathscr{H}_{m,n}.$
This means that every random variable $F\in  L^2(\Omega,\sigma(Z),P)$ admits a unique expansion of the type $F=\sum_{m=0}^\infty\sum_{n=0}^\infty F_{m,n}$, where $F_{m,n}\in \mathscr{H}_{m,n}, F_{0,0}=E[F]$ and the series converges in $ L^2(\Omega)$.
 \end{itemize}
\end{thm} There are  two essentially different approaches to show this theorem. The analysis approach makes use of Hahn-Banach theory and the Fourier transform or  Laplace transform along the same line of the proof of the real Wiener-It\^{o} chaos decomposition, please refer to 
\cite{ito, janson, np, Nualart} for details. The other approach makes use of the known real Wiener-It\^{o} chaos decomposition and the connection
between real Wiener-It\^o chaos and complex Wiener-It\^o chaos, which is an indirect proof but can be seen as an algebraic proof, please refer to  \cite{cl2} or Remark (\ref{newchaosre}).

The next goal of this section is Stroock's formula for the complex Wiener-It\^o chaos expansion. The starting point is to define complex Malliavin derivatives. 
\begin{defn}\it 
Let $\mathcal{S}$ denote the set of all random variables of the form
\begin{equation}\label{definition}
  f\big(Z(\varphi_1),\cdots,Z(\varphi_m)\big),
\end{equation}
where $f\in C_{\uparrow}^\infty(\Cnum^m)$ ($C^\infty $ function such that $f$and all its partial derivatives have polynomial growth order), $\varphi_i\in\FH, \ i=1,2,\cdots,m$ and $m\ge 1$.  A random variable belonging to $\mathcal{S}$ is said to be {\it smooth}.
If $F\in \mathcal{S}$, then the complex Malliavin derivatives of $F$ are the elements of $L^2(\Omega, \mathfrak{H} )$ defined by \cite{chw 17,cl2}:
 \begin{align}
    D  F     & =\sum_{i=1}^m \partial_{i } f(Z(\varphi_1),\dots,Z(\varphi_m))\varphi_{i},\\
    \bar{D}  F     & =\sum_{i=1}^m \bar{\partial}_{i } f(Z(\varphi_1),\dots,Z(\varphi_m))\bar{\varphi}_{i},
  \end{align} where
\begin{equation*}
   \partial_j f= \frac{\partial}{\partial z_j}f(z_1,\dots,z_m) ,\quad \bar{\partial}_j f=\frac{\partial}{\partial \bar{z}_j}f(z_1,\dots,z_m),\quad j=1,\dots,m
\end{equation*} are the Wirtinger derivatives.
\end{defn}

One can define the iteration of the operator  $ D$ and $\bar{D}$ in such a way that $D^p\bar{D}^q  F$ is a random variable with values in $\FH^{\odot p}\otimes \FH^{\odot q}$ for any $F\in \mathcal{S}$.
It is routine to show that $D^p\bar{D}^q $ are closable from $L^r(\Omega)$ to $L^r(\Omega,\FH^{\odot p}\otimes \FH^{\odot q})$ for every $r\ge 1$. Denote by $\DR^{p,r}\bigcap \bar{\DR}^{q,r}$ the closure of $\mathcal{S}$ with respect to the Sobolev seminorm $\norm{\cdot}_{p,q,r}$.  

\begin{thm}[Stroock's formula]
Every random variable $F\in  L^2(\Omega,\sigma(Z),P)$ can be expressed by $$F=\sum_{p=0}^\infty\sum_{q=0}^\infty I_{p,q}(f_{p,q}),$$ 
where $f_{p,q}\in \FH^{\odot {p}}\otimes \FH^{\odot q}$ . If $F\in \DR^{m,2}\bigcap \bar{\DR}^{n,2}$ then $$f_{p,q}=\frac{1}{p!q!}\E[D^p\bar{D}^q F],\qquad \forall p\le m,\,q\le n.$$
\end{thm}
\begin{proof}
The first part of the statement is a direct combination of Theorem~\ref{chaos} and the definition of  multiple Wiener-It\^{o} integrals.  For the second part, one implies from Proposition~\ref{prop d imn}  below that 
\begin{align*}
D^p\bar{D}^q  F=\sum_{i=p}^\infty\sum_{j=q}^\infty \frac{i!j! }{(i-p)!(j-q)!} I_{i-p,j-q}(f_{p,q}).
\end{align*}Taking expectation on both sides of the equality, one has that $f_{p,q}=\frac{1}{p!q!}\E[D^p\bar{D}^q F], \forall p\le m,\,q\le n $.
\end{proof}
\begin{prop}\label{prop d imn}
Suppose $ f \in \FH^{\odot m}\otimes \FH^{\odot n}$. Then one has that
\begin{align}
   D_{t}(I_{m,n}(f))&=m I_{m-1,n}(f(t,\cdot,\cdot)),\,\quad
   \bar{D}_{s}(I_{m,n}(f))=nI_{m,n-1}(f(\cdot,s,\cdot)).\label{d i mn}
   \end{align} %
\end{prop}
\begin{proof}
By the polarization technique, one can assume that $f=h^{\otimes m}\otimes \bar{h}^{\otimes n}$ with $h\in \FH$. Denote $\rho=\norm{h}^2$ and $\vec{t}^{\,k}=(t_1,\dots,t_k)$, $\vec{s}^{\,k}=(s_1,\dots,s_k)$. Then one obtains that
\begin{align*}
   D_{t}(I_{m,n}(f))&=D_{t}(J_{m,n}(Z(h),\rho))\\
   &=m J_{m-1,n}(Z(h),\rho)h(t)\\
   &=m I_{m-1,n}(h^{\otimes m-1}\otimes \bar{h}^{\otimes n})h(t)\\
   &=m I_{m-1,n}\big(f(t,\,\vec{t}^{m-1},\,\vec{s}^{\,n})\big).
\end{align*}  Similarly, one has that $\bar{D}_{s}(I_{m,n}(f))=nI_{m,n-1}\big(f(\vec{t}^{\,m},\,s,\,\vec{s}^{\,n-1})\big)$.
\end{proof}

Moreover,  three well known propositions concerning Malliavin derivatives are listed without proof to use later.
\begin{prop}[integration by parts formula]\label{prop 2_3}
   Suppose that $F\in\mathcal{S}$ and $h\in \FH$, then one has the following integration by parts formula
   \begin{align*}
        E[Z(h)\times \bar{F}]=E[\innp{h,\, D F}],\qquad  E[ \bar{Z} (h)\times \bar{F}]=E[\innp{h,\,\bar{D} F}]) .
   \end{align*}
\end{prop}
\begin{prop}\label{jprop}
   Let $\set{F_n,n\ge 1}$ be a sequence of random variable in $\DR^{1,2}$ (resp. $\bar{\DR}^{1,2}$) that converges to $F$ in $L^2(\Omega)$ and that
$$\sup\limits_nE[\norm{DF_n}_{\mathfrak{H}}^2]<\infty, \quad (\text{resp.  } \sup\limits_nE[\norm{\bar{D}F_n}_{\mathfrak{H}}^2]<\infty),$$
then $F $ belongs to $\DR^{1,2}$ (resp. $\bar{\DR}^{1,2}$) and the sequence of derivatives $DF_n$ (resp. $\bar{D}F_n$) converges weakly to $DF$ (resp. $\bar{D}F$) in $L^2(\Omega,\mathfrak{H})$.
\end{prop}

\begin{prop}{\bf (Chain rule)}\label{prop312}
 If $\varphi:\Cnum^m\to \Cnum$ is a continuously differentiable function with bounded partial derivatives and if $F=(F^1,\dots,F^m)$ is a random vector whose components are elements of $\DR^{1,2}\bigcap \bar{\DR}^{1,2}$, then $\varphi(F)\in \DR^{1,2}\bigcap \bar{\DR}^{1,2}$ and
\begin{align}
   D\varphi(F)&=\sum_{j=1}^m \partial_j \varphi(F)DF^j+\bar{\partial}_j \varphi(F)D\overline{ F^j},\label{chr1}\\
   \bar{D}\varphi(F)&=\sum_{j=1}^m \partial_j \varphi(F)\bar{D}F^j+\bar{\partial}_j \varphi(F)\bar{D}\overline{ F^j}.\label{chr2}
\end{align}
\end{prop}

\subsection{Product formula and Wick product}
In the spirit of Segal \cite{segal},  since the  Wiener-It\^o multiple integrals are isometric to the algebra of symmetric tensor, the product formula is crucial to the theory of complex multiple Wiener-It\^{o} integrals and implies some important properties.

\begin{thm}{\bf (Product formula)}\label{pdt fml}
For two symmetric functions $f\in \FH^{\odot a}\otimes \FH^{\odot b},\,g\in \FH^{\odot c}\otimes \FH^{\odot d}$,
the product formula for complex multiple Wiener-It\^{o} integrals is given by \cite{ch 17}
   \begin{align}\label{ii}
     I_{a,b}(f)I_{c,d}(g)&=\sum_{i=0}^{a\wedge d}\sum^{b\wedge c}_{j=0}\,{a\choose i}{d\choose i}{b\choose j}{c\choose j}i!j!\,I_{a+c-i-j,b+d-i-j}(f\otimes_{i,j}g).
   \end{align}
   where $a,b,c,d\in \Nnum$, and the contraction of $(i,j)$ indices of the two functions is given by
   \begin{align*}
   &  f\otimes_{i,j} g \\
    & =  \sum_{l_1,\dots,l_{i+j}=1}^{\infty} \innp{f,\,e_{l_1}\otimes \dots\otimes e_{l_i}\otimes \bar{e}_{l_{i+1}} \otimes \dots \otimes \bar{e}_{l_{i+j}}}\otimes \innp{g,\,{e}_{l_{i+1}}\otimes \dots\otimes {e}_{l_{i+j}}\otimes \bar{e}_{l_1}\otimes \dots\otimes \bar{e}_{l_{i}}}, 
     \end{align*}
by convention, $f\otimes_{0,0} g=f\otimes g$ denotes the tensor product of $f$ and $g$.
Especially, when $\FH=L^2(T,\mathcal{B},\nu)$ with $\nu$ non-atomic, one has that
\begin{align*}
& f\otimes_{i,j} g (t_1,\dots,t_{p_1+p_2-i-j
    };s_1,\dots,s_{q_1+q_2-i-j})\\ &=\int_{A^{i+j}}\nu^{i+j}(\dif u_1\cdots \dif u_i\dif v_1\cdots \dif v_j)\,f(t_1,\dots,t_{p_1-i},u_1,\dots,u_i; s_1\dots,s_{q_1-j},v_1\dots,v_j)\\
    & \times g(t_{p_1-i+1},\dots,t_{p_1-i+p_2-j},v_1\dots,v_j; s_{q_1-j+1}\dots,s_{q_1-j+q_2-i},u_1\dots,u_i). 
\end{align*}
\end{thm}
Hoshino et al. obtained a version of product formula for complex Wiener-It\^o integrals where $f,\, g$ can be non-symmetrized functions (see \cite[Theorem A.1]{hin 17}).

The product formula can be used to imply a  \"{U}st\"{u}nel-Zakai independence criterion for complex Wiener-It\^o integrals whose proof can be found in \cite{ch 17}.
Suppose $  f(t_1,\dots,t_p,s_1,\dots,s_q)\in \FH^{\odot p}\otimes \FH^{\odot q}$. We call $$ \FH^{\odot q}\otimes \FH^{\odot p}\ni h(t_1,\dots,t_q,s_1,\dots,s_p):=\bar{f}(s_1,\dots,s_p,t_1,\dots,t_q)$$   the reversed complex conjugate of function $  f(t_1,\dots,t_p,s_1,\dots,s_q)$.
Clearly,  one has that 
\begin{align}\label{re-cjgt}
 \overline{{I}_{p,q}(f)}={I}_{q,p}(h).
\end{align}

\begin{thm}\label{thm2}
   For two symmetric functions $f\in \FH^{\odot a}\otimes \FH^{\odot b},\,g\in \FH^{\odot c}\otimes \FH^{\odot d}$ with $a+b\ge 1 ,\,c+d\ge 1$, the following conditions are equivalent:
     \begin{itemize}
    \item[\textup{(i)}] $I_{a,b}(f)$ and $I_{c,d}(g)$ are independent random variables;
    \item[\textup{(ii)}] $f\otimes_{1,0}g=0$, $f\otimes_{0,1}g=0$, $f\otimes_{1,0}h=0$ and $f\otimes_{0,1}h=0$ in $\FH^{\otimes(a+b+c+d-2) }$, where $h$ is the reversed complex conjugate of $g$.
  \end{itemize}
\end{thm}
The product formula implies that a product of two complex multiple integrals is a finite sum of multiple integrals. This is close related to the Wick product which can be used to define the renormalization of the stochastic complex Ginzburg-Landau equation \cite{hin 17}. The Wick product is the last topic of this section.
\begin{definition}\it 
For $(m,n)\ge 0$, $\pi_{m,n}$ denotes the orthogonal projection of $L^2(\Omega)$ onto $\mathscr{H}_{m,n}$, and  $\pi_{\le(m,n)}$ the orthogonal projection of $L^2(\Omega)$ onto $\bigoplus^{m}_{i=0}\bigoplus^{n}_{j=0}\mathscr{H}_{i,j}$.
For any $h_1,\dots, h_{m+n}\in \FH$, the Wick product $:Z(h_1)\dots Z(h_m)\overline{Z({h_{m+1}})}\dots\overline{Z({h_{m+n}})}:$ is given by 
\begin{align*}
:Z(h_1)\dots Z(h_m)\overline{Z({h_{m+1}})}\dots\overline{Z({h_{m+n}})}:\,=\pi_{m,n} \big(Z(h_1)\dots Z(h_m)\overline{Z({h_{m+1}})}\dots\overline{Z({h_{m+n}})} \big).
\end{align*}
Define the general Wick product by 
\begin{align*}
\xi \diamond \eta=\pi_{m+p,n+q}(\xi \eta),
\end{align*}if $\xi\in \mathscr{H}_{m,n}$ and $\eta\in \mathscr{H}_{p,q}$, and extend $\diamond$ by bilinearity to a binary operator on the finite order chaos space  $\overline{\mathcal{P}}_{*}(\FH)=\sum_{m,n=0}^{\infty}   \mathscr{H}_{m,n}$.
\end{definition}
The product formula and the hypercontractivity inequality (see below subsection \ref{hyperctc}) imply that the general Wick product is a continuous bilinear operator $\mathscr{H}_{m,n}\times \mathscr{H}_{p,q}\to \mathscr{H}_{m+p,n+q} $. 
Thus there exists a constant $c(m,n,p,q)$ such that for any $\xi\in \bigoplus^{m}_{i=0}\bigoplus^{n}_{j=0} \mathscr{H}_{i,j}$ and $\eta\in \bigoplus^{p}_{i=0}\bigoplus^{q}_{j=0}\mathscr{H}_{i,j}$,
\begin{align*}
\norm{\xi \diamond \eta}_2\le c(m,n,p,q)\norm{\xi}_2\norm{\eta}_2.
\end{align*}
The constant $c(m,n,p,q)$ can not be replaced by a constant independent of $m,\,n,\,p$ and $q$. 

\begin{ex}
\it By (\ref{ii})(\ref{re-cjgt}) and induction,  for $f\in \FH$, $\big(I_{1,0}(f)\big)^n=I_{n,0}(f^{\odot n})$. Therefore, 
\begin{align}\label{wick-ex}
 :Z(f)^n: \ \ = I_{n,0}(f^{\odot n})=\big(Z(f)\big)^n.
\end{align}
\begin{align}\label{wtime}
 :Z(f)^p\overline{Z(f)}^q:\ \  =\sum_{k=0}^{p\wedge q}(-1)^k k! {p\choose k}{q\choose k}Z(f)^{p-k}\overline{Z(f)}^{q-k}(E|Z(f)|^2)^k.
\end{align}
Especially, let $\zeta$ be a symmetric complex gaussian variable, i.e. $\zeta=\xi+\mi \eta$ where $\xi$ and $\eta$ are independent centered Gaussian variable with the same variance, and take $\FH=\mathbb{C}$, then for $f=f_1+\mi f_2\in \mathbb{C}$, $Z(f)=(f_1\xi-f_2\eta)+\mi(f_1\eta+f_2\xi)$ is also a symmetric complex Gaussian variable 
and (\ref{wick-ex}) (\ref{wtime}) coincides with the wick powers of a symmetric complex Gaussian variable in Example 3.30 and Example 3.31 respectively in \cite{janson} which are induced by Feynman diagram.  

In fact, along the same ways in section III.2 in \cite{janson}, one can define the wick exponential $:e^{Z(f)}:$ by $e^{Z(f)}$. 
\end{ex}


\subsection{The Hu-Meyer formula}
The Hu-Meyer formula relates the multiple Stratonovich integral to the multiple Wiener-It\^{o} integral and vice versa (see \cite{hm 88}\cite{hm 93}\cite{hyz 17}). 

The product formula (\ref{ii}) can also be used to imply the complex version of Hu-Meyer formula. 
 Take a complete orthonormal system $\set{ {e}_k}$ in $\mathfrak{H} $.  Let $f\in \FH^{\odot p }\otimes \FH^{\odot q }$ and consider the following random variable:
\begin{align}\label{sn}
S^n_{p,q}(f)=\sum_{i_1,\dots,i_p=1}^{n}\sum_{l_1,\dots,l_q=1}^{n}\,\innp{f,\,(e_{i_1}\hat{\otimes}\dots\hat{\otimes}e_{i_p})\otimes (\bar{e}_{l_1}\hat{\otimes}\dots\hat{\otimes}\bar{e}_{l_q})} Z(e_{i_1})\dots Z(e_{i_p})\overline{Z({e_{l_1}})}\dots\overline{Z({e_{l_q}})}.
\end{align}
If the limit in probability of $S^n_{p,q}(f)$ exists as $n\to \infty$, one calls  $f$ is Stratonovich integrable. The limit is called  the multiple Stratonovich integral of $f$
and is denoted by $S_{p,q}(f) $.
\begin{definition}\it 
Suppose that $0\le k\le  p \wedge q$. Denote that 
\begin{align*}
\mathrm{Tr}^{k,n} f
&=\sum_{i_1,\dots,i_{p+q-k}=1}^{n} \innp{f,\,(e_{i_1}\hat{\otimes}\dots \hat{\otimes} e_{i_k} \hat{\otimes} \dots \hat{\otimes} e_{i_{p}})\otimes(\bar{e}_{i_1}\hat{\otimes}\dots \hat{\otimes} \bar{e}_{i_k}\hat{\otimes}\bar{e}_{i_{p+1}}\hat{\otimes}\dots\hat{\otimes} \bar{e}_{i_{p+q-k}})}\\
&\qquad \times (e_{i_{k+1}}\hat{\otimes} \dots \hat{\otimes} e_{i_{p}})\otimes (\bar{e}_{i_{p+1}}\hat{\otimes}\dots\hat{\otimes} \bar{e}_{i_{p+q-k}}). 
\end{align*}
If $\mathrm{Tr}^{k,n} f$ converges in $\FH^{\odot (p-k) }\otimes \FH^{\odot (q-k) }$ as $n\to \infty$, then one says that $f$ has a trace of order $k$ and the limit  is denoted by $\mathrm{Tr}^{k} f$.
\end{definition}
\begin{thm}[The Hu-Meyer formula]
Suppose $f\in \FH^{\odot p }\otimes \FH^{\odot q }$. If the traces of order $k$ of $f$ exist for all $k\le (p \wedge q) $, then $f$ is Stratonovich integrable and 
\begin{align}
S_{p,q}(f)&=\sum_{k=0}^{p\wedge q}  k! {p \choose k}{q \choose k} I_{p-k,\,q-k}(\mathrm{Tr}^k f); \label{humey 1}\\
I_{p,q}(f)&=\sum_{k=0}^{p\wedge q} (-1)^k k!{p \choose k}{q \choose k}S_{p-k,\,q-k}(\mathrm{Tr}^k f ).\label{humey 2}
\end{align}
\end{thm}
\begin{proof}
The product formula implies that
\begin{align*}
& Z(e_{i_1})\dots Z(e_{i_p})\overline{Z({e_{l_1}})}\dots\overline{Z({e_{l_q}})}\\
&=I_{p, 0}(e_{i_1}\hat{\otimes} \dots\hat{\otimes}e_{i_p})I_{0,q}(\bar{e}_{l_1}\hat{\otimes} \dots\hat{\otimes}\bar{e}_{l_q})\\
&=\sum_{k=0}^{p\wedge q} {p \choose k}{q \choose k} k! I_{p-k,q-k}\big((e_{i_1}\hat{\otimes} \dots\hat{\otimes}e_{i_p})\otimes_{k,k} (\bar{e}_{l_1}\hat{\otimes} \dots\hat{\otimes}\bar{e}_{l_q})\big)\\
&=\sum_{k=0}^{p\wedge q} {p \choose k}{q \choose k} k! I_{p-k,q-k}\big((e_{i_{k+1}}\hat{\otimes} \dots\hat{\otimes}e_{i_{p}})\otimes (\bar{e}_{l_{k+1}}\hat{\otimes} \dots\hat{\otimes}\bar{e}_{l_{q}} ) \big) \delta_{i_1,l_1}\cdots \delta_{i_k,l_k},
\end{align*}
where $\delta_{i,j}$ is the Kronecker delta. By substituting the above equation displayed into Eq. (\ref{sn}), we obtain that 
\begin{align*}
S^n_{p,q}(f)=\sum_{k=0}^{p\wedge q}  k! {p \choose k}{q \choose k} I_{p-k,\,q-k}(\mathrm{Tr}^{k,n} f).
\end{align*}Taking the limit, one implies that the formula (\ref{humey 1}) holds.

By the linear property and the polarization technique, to show (\ref{humey 2}), it suffices to show it for the case of $f=h^{\otimes p}\otimes \bar{h}^{\otimes q}$ with $\norm{h}=\sqrt2$.  In this case, (\ref{humey 2}) is degenerated to the identity (\ref{jal}). 
\end{proof}


\subsection{The relation between the real and complex Wiener-It\^o integrals}

The real and imaginary parts of a complex multiple Wiener-It\^o integral of $(m,\,n)$-order can
be expressed as real integrals of order $m+n$ \cite{cl2}. In details, 
denote by $\FH_{\R}$ the real Hilbert space such that $\FH=\FH_{\R}+\mi \FH_{\R}$ and
suppose that  $X=\set{X(h):\,h\in \FH_{\Rnum}}$ is an isonormal Gaussian process over $\FH_{\Rnum}$ and that 
$Y$ is an independent copy of $X$. Clearly, a realization of the isonormal Gaussian process $W$ over the Hilbert space direct sum space $\FH_{\Rnum}\oplus \FH_{\Rnum}$ is
\begin{align*}
      W(h,f)&=X(h)+Y(f),\qquad\forall h,f\in\FH_{\Rnum}.
   \end{align*}
Then one obtains that \cite{cl2}
\begin{thm}\label{pp2} 
Suppose $\varphi\in \FH^{\odot m}\otimes \FH^{\odot n}$ and $F={I}_{m,n}(\varphi)=F_1+\mi F_2$. Then
     there exist real $u,\,v\in (\FH_{\R}\oplus \FH_{\R})^{\odot (m+n)}$ such that
      \begin{align}
        F_1&=\mathcal{I}_{m+n}(u),\quad F_2=\mathcal{I}_{m+n}(v),\label{eeq}
      \end{align}
      where $\mathcal{I}_p(g)$ is the $p$-th real Wiener-It\^{o} multiple integral of $g$ with respect to $W$. 
\end{thm}
Beside being used to show Theorem~\ref{chaos}, the above theorem has other applications such as to show complex fourth moment theorems. It is implied from the following proposition \cite{cl2}.
\begin{prop}\label{prop 3.5}
Suppose that $\norm{f}_{\FH_{\R}}^2+\norm{g}_{\FH_{\R}}^2=1$,  then  for any fixed $\theta\in \Rnum$,
    \begin{align*} 
   H_n\big(X(f) +Y(g)\big) = \sum_{k=0}^n \, d_k  J_{k,n-k}(Z({h})),
\end{align*} where 
${h}=\sqrt2 e^{\mi \theta}(f-\mi g)$,
and
\begin{align*} 
   d_k=\frac{1}{2^n}\sum_{r+s=k}(-1)^s \sum_{l=0}^n {n \choose l}{l \choose r}{n-l\choose s}(\cos\theta)^l(\mi \cdot\sin\theta)^{n-l}.
\end{align*}
Suppose that $\FH \ni {h} $ with $\norm{ {h}}_{\FH }=\sqrt{2}$, then
\begin{equation*} 
  J_{k,n-k}(Z( {h}))= \sum_{i=0}^n\tilde{c}_i H_n(X(f_i)+Y(g_i)),
\end{equation*}
where 
$f_i+\mi g_i=\frac{1}{\sqrt2} e^{\mi \theta_i}\bar{ {h}}$,
and
\begin{equation*}  
   \tilde{c}_i=\sum_{j=0}^n\tensor{M}^{-1}_{j,i}{\mi^{n-j}}\sum_{r+s=j}{k \choose r}{n-k \choose s}(-1)^{n-k-s},
\end{equation*} where the matrix $M$ is given as in (\ref{matr}). That is to say
\begin{align*}
 \mathcal{H}_n(W)=\bigoplus_{k+l=n}\mathcal{H}_k(X)\mathcal{H}_l(Y),
\end{align*}
\begin{align*}
  \mathcal{H}_n(W)+\mi \mathcal{H}_n(W)=\bigoplus_{k+l=n}\mathscr{H}_{k,l},
\end{align*}
\begin{align}\label{newchaos}
  L^2(\Omega,\sigma(Z),P)=\bigoplus^{\infty}_{n=0} \big(\mathcal{H}_n(W)+\mi \mathcal{H}_n(W)\big)=\bigoplus^{\infty}_{n=0}\bigoplus_{k+l=n}\mathscr{H}_{k,l}=\bigoplus^{\infty}_{m=0}\bigoplus^{\infty}_{n=0}\mathscr{H}_{m,n}.
 \end{align}
Where $ \mathcal{H}_n(W),  \mathcal{H}_n(X)$ and $ \mathcal{H}_n(Y)$ are the $n$-th Wiener-It\^o Chaos with respect to $W,X$ and $Y$ respectively.

\end{prop}

\begin{remark}\label{newchaosre}
   On the one hand, the first equality in (\ref{newchaos}) may be regarded as a new chaos decomposition of complex Gaussian systems in spirit of Wiener, since  $L^2$ function is expanded by a series of polynomials of  Gaussian random vector $W=(X,Y)$.  On the other hand, along the equality from the left side to the right side in (\ref{newchaos}), this equality provides a way to prove the chaos decomposition in sense of It\^o (see section 4 in \cite{cl2} for the details of this proof).  
   
 This decomposition can be applied to show the Gaussian and the non-Gaussian center limit theorem of complex Wiener-It\^o multiple integrals (see \cite{cl2}).
   
\end{remark}


\section{Complex Ornstein-Uhlenbeck operators and semigroups}


In this section,  one obtains several analytical properties of multiple Wiener-It\^{o} integrals such as their norm equivalence and exponential integrability, and one expands the fourth moment of multiple Wiener-It\^{o} integrals in terms of their Malliavin derivatives and  the contractions of kernels of multiple Wiener-It\^{o} integrals. The method is by means of Malliavin divergence operators and Ornstein-Uhlenbeck operators. 
\subsection{Malliavin divergence operators and Clark-Ocone formula}

The divergence operators $\delta$ and $\bar{\delta} $ are defined as the adjoint of $D$ and $\bar{D}$ respectively, with the domains $\mathrm{Dom}( \delta) $ and $\mathrm{Dom}( \bar{\delta}) $ the subsets of $L^2(\Omega,\FH )$ composed of those elements $u$ such that there exists a constant $c>0$ verifying for all $ F\in \mathcal{S}$,
   \begin{align*}
     \big|E[\innp{D F,u}]\big|\le c\norm{F}, \quad(\text{resp. } \big|E[\innp{\bar{D} F,u}]\big|\le c\norm{F}).
   \end{align*}
   If $u\in\mathrm{Dom}( \delta ) $ or $u\in\mathrm{Dom}( \bar{\delta} )$, then $ \delta  u$ and $\bar{\delta}u$ are the unique element of $L^2(\Omega)$ given respectively by the following duality formula: for all $ F\in \mathcal{S}$,
   \begin{align}\label{dul form}
      E[(\delta  u)\times \bar{F}]=E[\innp{u,\, D F}],\quad \big(\text{resp. }
      E[( \bar{\delta}  u)\times \bar{F}]=E[\innp{u,\,\bar{D} F}]\big) .
   \end{align}

Consider a special case that  $\set{Z_t,\,t\ge 0}$ is a  complex one-dimensional fractional Brownian motion with a fixed Hurst index $H\in (0,1)$.
Let $L^2_a$ be the set of square integrable random variables adapted to the completed sigma-field filter $\mathcal{F}_t=\sigma(Z_s,\,s\le t)$. 
By the chaos expansion,  one has a Clark type representation of a functional of fractional Brownian motion. The proof is the same as the real case, please refer to \cite{Hu 05} for details.
\begin{prop}[Clark type representation]\label{itorep}
 Let random variable $F\in  L^2(\Omega,\sigma(Z),P)$. Then there exist two stochastic processes 
 $\varphi(t,\omega)\in L^2_a$ and $\psi(t,\omega)\in L^2_a$ almost everywhere uniquely such that 
\begin{equation}
   F=\E[F]+\int_0^{\infty}\varphi(t,\omega)\dif Z_t+\int_0^{\infty}\psi(t,\omega)\dif \bar{Z}_t ,
\end{equation} where the integral is a divergence integral.
\end{prop}
It is routine to extend $\set{Z_t,\,t\ge 0}$ to a complex Gaussian isonormal process $Z=\set{Z(h):\,h\in\FH}$ for a suitable Hilbert space $\FH$.
\begin{thm}[Clark-Ocone formula]\label{cof}
Suppose that $\set{Z_t,t\ge 0}$ is a complex one-dimensional fractional Brownian motion with a fixed Hurst index $H\in (0,1)$. 
 If $F\in \DR^{1,2} \bigcap \bar{\DR}^{1,2} $, then
\begin{equation}
  F=\E[F]+ \int_0^{\infty}\E(D_t F|\mathcal{F}_t)\dif Z_t+\int_0^{\infty}\E(\bar{D}_t F|\mathcal{F}_t)\dif \bar{Z}_t,
\end{equation}  where the integral is a divergence integral.
\end{thm}
\begin{proof} For simplicity, one suppose that $H\in (\frac12, \,1)$. The other case is similar. 
Denote $\alpha_H=H(2H-1)$ and  $\phi(s,t) =\alpha_H \abs{s-t}^{2H-2}$ and define the Hilbert space
\begin{align*}
   \FH&:=L^2_{\phi}=\{f |\,f:\R_{+}\to \Cnum,\abs{f}_{\phi}^2:=\int_0^{\infty}\int_0^{\infty}f(s)\bar{f}(t)\phi(s,t)\dif s\dif t <\infty \}.
\end{align*}

For any element $u\in L^2_a$, by It\^{o}'s isometry and Proposition~\ref{itorep}, one has that
\begin{align*}
    E[(\delta  u)\times \bar{F}]=\int_0^{\infty}\int_0^{\infty} \E[u_t \bar{\varphi}_s]\phi(s,t) \dif t\dif s,
\end{align*}
By the duality formula, one has that 
\begin{align*}
   E[(\delta  u)\times \bar{F}]&=\int_0^{\infty}\int_0^{\infty} \E[u_t \overline{D_s F}]\phi(s,t)\dif t\dif s \nonumber \\
   &=\int_0^{\infty} \int_0^{\infty} \E[u_t \overline{\E[D_s F|\mathcal{F}_s]}]\phi(s,t) \dif t \dif s.
\end{align*}
The above two equations imply that $\varphi_s= \E[D_s F|\mathcal{F}_s]$.
In the same way, one can show that $\psi_s=\E[\bar{D}_s F|\mathcal{F}_s] $.
\end{proof}

\subsection{Complex Ornstein-Uhlenbeck operators and the hypercontractivity of complex  Ornstein-Uhlenbeck semigroup }\label{hyperctc}

\begin{definition}\it 
   Complex Ornstein-Uhlenbeck operators are defined as
\begin{align}
   \tensor{L}=\delta D,\qquad \bar{\tensor{L}}=\bar{\delta} \bar{D}.
\end{align}
\end{definition}

\begin{prop}\label{prop ll imn 2}
Suppose that ${I}_{m,n}(f)$ is the complex Wiener-It\^{o} integral of $f$ with respect to $Z$ for $ f \in \FH^{\odot m}\otimes \FH^{\odot n}$. Then one has that
\begin{align}
   \tensor{L}(I_{m,n}(f) )&=m I_{m,n}(f),\qquad \bar{\tensor{L}}(I_{m,n}(f) )=n I_{m,n}(f).\label{l i mn}
\end{align} %
\end{prop}
\begin{proof} By the polarization technique, one need only to show that (\ref{l i mn}) holds in case of $f=h^{\otimes m}\otimes \bar{h}^{\otimes n}$.
Let $G=I_{m-1,n}( h^{\otimes m-1}\otimes \bar{h}^{\otimes n})$ and $\rho=\norm{h}^2$.  One has that 
\begin{align*}
   \tensor{L}(I_{m,n}(f) )&=m \delta (G h )\\
   &=m [GZ(h)-\innp{h,\,{D}\bar {G}}_{\FH}]\\
   &=m[Z(h) J_{m-1,n}(Z(h),\rho)-n\rho J_{m-1,n-1}(Z(h),\rho)]\\
   &=m J_{m,n}(Z(h),\rho)\\
   &=mI_{m,n}(f).
\end{align*}
Similarly, one has that $\bar{\tensor{L}}(I_{m,n}(f) )=n I_{m,n}(f)$.
\end{proof}

\begin{definition}\it 
Fix a $\theta\in (-\frac{\pi}{2},\,\frac{\pi}{2})$.
The OU semigroup is the one-parameter semigroup $\set{T_t:\,t\ge 0}$ of contraction operators on $L^2(\Omega,\sigma(Z),P)$ defined by 
\begin{align*}
T_t (F) = \sum_{m,n=0}^{\infty} e^{-[(m+n)\cos \theta+\mi (m-n)\sin \theta]t}I_{m,n}(f_{m,n}),
\end{align*}
where $F$ is given by $F=\sum_{m=0}^\infty\sum_{n=0}^\infty I_{m,n}(f_{m,n})$ with $f_{m,n}\in \FH^{\odot {m}}\otimes \FH^{\odot n}$.
\end{definition}
It is clear that the infinitesimal generator of the semigroup $\set{T_t}$ is given by 
\begin{align*}
L_{\theta}=-(e^{\mi\theta}\tensor{ L}+e^{-\mi\theta}\bar{\tensor {L}}).
\end{align*}
The Mehler's formula is as follows \cite{ch 15}.
\begin{prop}
Let $r=e^{\mi\theta}$ and $Z'=\set{Z'(h):\,h\in \FH}$ be an independent copy of $Z$. Then, for any $F\in L^2(\Omega)$,
\begin{align*}
T_t(F)(Z)=\E_{Z'}[F(e^{-rt}Z+\sqrt{1-e^{-2t \cos\theta }} Z') ],\qquad t\ge 0.
\end{align*}
\end{prop}

\begin{ques}\it 
Since the operator $\tensor{ L}$ is not the generator of  the semigroup $\set{T_t}$, one does not have a version of Mehler's formula for $\tensor{L}$. Thus, one can not copy Pisier's proof to obtain the following  Meyer's inequalities:
\begin{align*}
c_p\norm{DF}_p\le \norm{(-L)^{1/2}F}_p\le C_p\norm{DF}_p,\qquad \forall F\in \DR^{1,p},
\end{align*}where $p>1$ and $c_p,\,C_p$ are two positive constants. In fact, it is not  known whether the Meyer's inequalities hold or not.
\end{ques}

The OU semigroup $\set{T_t}$ verifies a property called Hypercontractivity \cite{ch 15}.
\begin{prop}
For the fixed $t \ge 0$ and $p> 1$, set $q(t) =e^{2t \cos\theta}(p-1)+1$. Then 
\begin{align*}
\norm{T_t F}_{q(t)}\le \norm{F},\qquad \forall F\in L^p(\Omega). 
\end{align*}
\end{prop}
A direct consequence of the hypercontractivity property is the norm equivalence of the  Wiener-It\^{o} chaos \cite{ch 17}. 
\begin{prop}\label{pp201}
    \begin{itemize}
        \item[\textup{(1)}] Complex multiple Wiener-It\^{o} integrals have all moments satisfying the following hypercontractivity inequality
        \begin{equation}\label{hypercontract}
         [E\abs{I_{p,q}(f)}^{r}]^{\frac{1}{r}}\le (r-1)^{\frac{p+q}{2}} [E\abs{I_{p,q}(f)}^{2}]^{\frac{1}{2}},\quad r\ge 2,
        \end{equation}
         where $\abs{\cdot}$ is the absolute value (or modulus) of a complex number.
        \item[\textup{(2)}] If a sequence of distributions of $\set{I_{p,q}(f_n)}_{n\ge 1}$ is tight, then
        \begin{equation}\label{tight compact}
          \sup_{n} E\abs{I_{p,q}(f_n)}^{r}<\infty \quad\text {for every $r>0$.}
        \end{equation}
    \end{itemize}
\end{prop}
The inequality (\ref{hypercontract}) implies that the exponential integrability of $I_{p,q}(f)$. For the real case, please refer to \cite{hyz 17}.
\begin{corollary}
For the multiple Wiener-It\^{o} integral $I_{p,q}(f)$,
there exist a constant $c_0>0$ such that for all $c<c_0$,
\begin{align*}
\E\exp\set{ c_0 \abs{I_{p,q}(f)}^{\frac{2}{p+q}}}<\infty.
\end{align*}
\end{corollary}

\subsection{Several expansions of the fourth moment of a multiple Wiener-It\^o integral}
The moment method is a crucial tool of probability. The formulae of the moments of real multiple Wiener-It\^o  integrals show some interesting combinatorial structures (see \cite{pt}).   However, it seems that the formulae of the moment of complex multiple Wiener-It\^o  integrals is much more complicated than the real one. 

The first expansion of the fourth moment of a complex multiple Wiener-It\^o integral is in terms of Malliavin derivatives. It offers a starting point for this topic.

\begin{prop}Suppose that $F=I_{m,n}(f)$ with $f\in \FH^{\odot m}\otimes \FH^{\odot n}$,  then one has that
\begin{equation}\label{fourth moment} 
 \E[\abs{F}^4]=\frac{1}{m}\big[ 2 \norm{DF}_{\FH}^2\times\abs{F}^2 +\innp{DF,\,D\bar{F}}_{\FH} \times\bar{F}^2\big].
\end{equation}
\end{prop}
\begin{proof}
Step 1: By approximation, one claims that for any Wiener-Ito integral $F=I_{m,n}(f)$,
\begin{align}\label{qiu dao}
  D(\bar{F}F^2)=2\abs{F}^2 DF+F^2 D\bar{F}.
\end{align}
In fact, for the function $g(z)=\bar{z}z^2$ and $n\in \Nnum$, we take $$g_n=g\cdot \big[\chi_{[-n,n]}+k(-n-x)+k(n+x)\big]$$
 where $\chi_A(\cdot)$ the index function of a set $A$ and $k(x)=e^{-\frac{1}{x(1-x)}}\chi_{(0,1)}(x)$ a cut-off function. For any $p\ge 1$, $g_n\in C^{\infty}_c(\Rnum^2)$ and $g_n,\partial g_n,\,{\partial}\bar{g}_n$ converge to $g,\,\partial g,\,{\partial}\bar{g}$ respectively in the sense of $L^p(\mu)$ with $F\sim \mu$. The chain rule, i.e, Proposition~\ref{prop312}, implies that
 \begin{align*}
   D(g_n(F))=\partial g_n(F) DF+\bar{\partial} g_n(F) D\bar{F}.
 \end{align*} The hypercontrativity inequality of Wiener-Ito integrals (see Proposition~2.4 of \cite{ch 17}) and the Cauchy-Schwarz inequality imply that as $n\to \infty$, in the sense of $L^2(\Omega,\,\FH)$,
 \begin{align*}
   \partial g_n(F) DF+\bar{\partial} g_n(F) D\bar{F} \to 2\abs{F}^2 DF+F^2 D\bar{F}.
 \end{align*}
Then we obtain (\ref{qiu dao}) by Proposition~\ref{jprop}.

Step 2: It follows from Proposition~\ref{prop ll imn 2}, the dual relation and the chain rule that
\begin{align*}
  \E[\abs{F}^4]&=\E[ {F}\overline{\bar{F}F^2}]\\
  &=\frac{1}{m}\E[ \delta D {F}\times \overline{\bar{F}F^2}]\\
  &=\frac{1}{m}\E[ \innp{ D {F},\, D(\bar{F}F^2)}_{\FH}]\\
  &=\frac{1}{m}\big[ 2 \norm{DF}_{\FH}^2\times\abs{F}^2 +\innp{DF,\,D\bar{F}}_{\FH} \times\bar{F}^2\big].
\end{align*}
\end{proof}
The other three expansions of the fourth moment of a multiple Wiener-It\^o integral are in terms of the contraction of kernels of multiple Wiener-It\^o integrals.
\begin{prop}\label{prop 3-2}
Suppose that $F=I_{m,n}(f)$ with $f\in \FH^{\odot m}\otimes \FH^{\odot n}$ and that $\bar{F}=I_{n,m}(h)$ where $h$ is the reversed complex conjugate of $f$. Denote  $l=m+n,\,l'=2(m\wedge n)$. Then
\begin{align}
   & \E[\abs{F}^4]-2\big(\E[\abs{F}^2]\big)^2 -\abs{E[{F}^2]}^2 \nonumber\\
      &=\sum_{0 <i+j<l'}{m\choose i}{n\choose i}{n\choose j}{m\choose j}  (m!n!)^2  \norm{f\otimes_{i,j}f}^2_{\FH^{\otimes(2(l'-i-j))}} +\sum_{r=1}^{l-1}((l-r)!)^2\norm{\psi_r}^2_{\FH^{\otimes(2(l-r))}} \label{df s1}\\
      &=\sum_{0<i+j<l}{m\choose i}^2{n\choose j}^2(m!n!)^2 \norm{f\otimes_{i,j}h}^2_{\FH^{\otimes(2(l-r))}}  +\sum_{r= 1}^{l'-1} (2m-r)!(2n-r)!  \norm{\varphi_r}^2_{\FH^{\otimes 2(l'-r)}},\label{df s2}\\
      &=2\sum_{r=1}^{l-1}\big[(l-r)!\big]^2
   \innp{\vartheta_r,\psi_r}_{\FH^{\otimes 2(l-r)}}+ \sum_{r=1}^{l'-1}(2m-r)!(2n-r)!\innp{\varsigma_r,\, \varphi_r}_{\FH^{\otimes 2(l'-r)}}, \label{df ff1}
\end{align}
where 
\begin{align*}
   \vartheta_r
     &= \sum_{i+j=r}\frac{i}{m} {m \choose i }^2{n\choose j}^2 i!j!\,f\tilde{\otimes}_{i,j}h,\\
\psi_r&=\sum_{i+j=r} {m\choose i}^2{n\choose j}^2 i!j!\, f\tilde{\otimes}_{i,j} h,  \\
    \varsigma_r&= \sum_{i+j=r}\frac{i}{m} {m \choose i }{n\choose i}{m \choose j }{n\choose j} i!j!\,f\tilde{\otimes}_{i,j}f,\\
   \varphi_r&=\sum_{i+j=r} {m\choose i}{n\choose i}{n\choose j}{m\choose j} i!j!\, f\tilde{\otimes}_{i,j} f.
\end{align*}
\end{prop}
\begin{proof} One need only to show the third identity since the first and the second are shown in \cite{chw 17}. Its proof is divided into two steps.

  Step 1: One claims that
  \begin{align}
   \frac{1}{m}\E\big[\abs{F}^2\norm{D F}^2_{\FH}\big]=\big(\E[\abs{F}^2]\big)^2+\sum_{r=1}^{m+n-1}\big[(m+n-r)!\big]^2
   \innp{\vartheta_r,\psi_r}_{\FH^{\otimes 2(m+n-r)}} .\label{first part}
  \end{align}
  In fact, it follows from the product formula and the Fubini theorem that
  \begin{align}
    \frac{1}{m}\norm{D F}^2_{\FH}&=m \norm{I_{m-1,n}(f)}^2_{\FH}\nonumber \\
    &=m \sum_{i=0}^{m-1}\sum^{n}_{j=0}\,{m-1\choose i}^2{n\choose j}^2 i!j!\,I_{m+n-1-i-j,m+n-1-i-j}(f\otimes_{i+1,j}h)\nonumber \\
    &=m \sum_{i=1}^{m}\sum^{n}_{j=0}\,{m-1\choose i-1}^2{n\choose j}^2 (i-1)!j!\,I_{m+n-i-j,m+n-i-j}(f\otimes_{i,j}h)\nonumber \\
    &=\E[\abs{F}^2]+\sum_{r=1}^{m+n-1}I_{m+n-r,m+n-r}(\vartheta_r).\label{df square1}
  \end{align}
  On the other hand, one can obtain that
  \begin{align}\label{f squre}
     \abs{F}^2&=\sum_{i=0}^{m}\sum^{n}_{j=0}\,{m\choose i}^2{n\choose j}^2 i!j!\,I_{m+n-i-j,m+n-i-j}(f\otimes_{i,j}h)\nonumber\\
     &=\E[\abs{F}^2]+\sum_{r=0}^{m+n-1}I_{m+n-r,m+n-r}(\psi_r).
  \end{align}
  Substituting (\ref{df square1}) and (\ref{f squre}) into the left side of (\ref{first part}) and using the orthogonality properties of multiple integrals, one has that (\ref{first part}) holds.

  Step 2: One claims that
  \begin{align}\label{second part}
    \frac{1}{m}\E\big[\innp{D F,\,D\bar{F}}_{\FH}\bar{F}^2\big]= \abs{\E[{F}^2]}^2+\sum_{r=1}^{2(m\wedge n)-1}(2m-r)!(2n-r)!\innp{\varsigma_r,\, \varphi_r}_{\FH^{\otimes 2(m+n-r)}}.
  \end{align}
  In fact, the product formula and the Fubini theorem implies that
  \begin{align}
   \frac{1}{m}\E\big[\innp{D F,\,D\bar{F}}_{\FH}&=n\innp{I_{m-1,n}(f),\,I_{n-1,m}(h)}_{\FH}  \nonumber\\
   &=n \sum_{i=0}^{m\wedge n-1}\sum^{m\wedge n}_{j=0}\,{m-1\choose i}{n-1\choose i}{m\choose j}{n\choose j}i!j!\,I_{2m-1-i-j,2n-1-i-j}(f\otimes_{i+1,j}f)\nonumber\\
   &=n \sum_{i=1}^{m\wedge n }\sum^{m\wedge n}_{j=0}\,{m-1\choose i-1}{n-1\choose i-1}{m\choose j}{n\choose j}(i-1)!j!\,I_{2m-i-j,2n -i-j}(f\otimes_{i,j}f)\nonumber\\
   &=\E[{F}^2] +\sum_{r=1}^{2(m\wedge n)-1}I_{2m-r,2n-r}(\varsigma_r).\label{df square12}
  \end{align}
  On the other hand, one can obtain that
  \begin{align}
     F^2&=\sum_{i,j=0}^{m\wedge n}\,{m\choose i}{n\choose i}{m\choose j}{n\choose j}i!j!\,I_{2m-i-j,2n-i-j}(f\otimes_{i,j}f)\nonumber\\
     &=\E[{F}^2]+\sum_{r=1}^{2(m\wedge n)-1}I_{2m-r,2n-r}(\varphi_r).\label{f squre 2}
  \end{align}
  Substituting (\ref{df square12}) and (\ref{f squre 2}) into the left side of (\ref{second part}) and using the orthogonality properties of multiple integrals, one has that (\ref{second part}) holds.

Finally, by substituting (\ref{first part}) and (\ref{second part}) into the equality (\ref{fourth moment}), one obtains the identity (\ref{df ff1}).
\end{proof}

\begin{corollary}\label{cor 3-12}
Under the assumption of Proposition~\ref {prop 3-2}, there exist two positive constants $c_1(m,n)$ and $c_2(m,n)$ such that the following inequalities hold:
\begin{align}
&\quad  c_1(m,n) \Big( \sum_{0<i+j<l}\norm{f\otimes_{i,j}h}^2_{\FH^{\otimes(2(l-r))}}  + \sum_{0<i+j<l'}\norm{f\otimes_{i,j}f}^2_{\FH^{\otimes(2(l'-r))}} \Big)\nonumber \\
&\le \E[\abs{F}^4]-2\big(\E[\abs{F}^2]\big)^2 -\abs{E[{F}^2]}^2 \label{ineq 1}\\ 
&\le c_2(m,n) \Big( \sum_{0<i+j<l}\norm{f\tilde{\otimes}_{i,j}h}^2_{\FH^{\otimes(2(l-r))}}  + \sum_{0<i+j<l'}\norm{f\tilde{\otimes}_{i,j}f}^2_{\FH^{\otimes(2(l'-r))}} \Big)\label{ineq 2}
\end{align}
\end{corollary}
\begin{proof}
 The inequality (\ref{ineq 1}) is implied from the identities (\ref{df s1}) and (\ref{df s2}). The inequality (\ref{ineq 2}) is implied from the identity (\ref{df ff1}), the Cauchy-Schwarz inequality and Minkowski's inequality.
\end{proof}
As an application of Corollary~\ref{cor 3-12}, one can give a direct and short proof for the equivalence between the following conditions $(\mathrm{iii})$ and $(\mathrm{iv})$, i.e.,  the convergence of the symmetrized contraction norms and that of the non-symmetrized contraction norms are equivalent.

\begin{thm}[Fourth Moment Theorems]\label{equilent cond}
Let   $\set{F_{k}=I_{m,n}(f_k)}$ with $f_k\in \FH^{\odot m}\otimes \FH^{\odot n}$  be  a  sequence of  $(m,n)$-th complex Wiener-It\^{o} multiple integrals, with $m$ and $n$ fixed and
$m+n\ge 2$.  Suppose that as $k\to \infty$, $E[\abs{F_k}^2]\to \sigma^2$ and $E[F_k^2]\to c+\mi b $, where $\abs{\cdot}$ is the absolute value (or modulus) of a complex number and  $c,b\in \Rnum$. Then   the following statements  are equivalent:
\begin{itemize}
\item[\textup{(i)}]The sequence $(\RE F_k,\,\IM  F_k)$ converges in law to a bivariate normal distribution with covariance matrix $\tensor{C}= \frac{1}{2}\begin{bmatrix}\sigma^2+c & b\\ b & \sigma^2-c  \end{bmatrix}$,
\item[\textup{(ii)}]$E[\abs{F_k}^4]\to c^2+b^2+2\sigma^4$.
\item[\textup{(iii)}] $\norm{f_k\otimes_{i,j} f_k}_{\FH^{\otimes ( 2(l-i-j))}}\to 0$ and $\norm{f_k{\otimes}_{i,j} h_k}_{\FH^{\otimes ( 2(l-i-j))}}\to 0 $ for any $0<i+j\le l-1$ where $l=m+n$ and $h_k$ is the kernel of $\bar{F}_k$, i.e., $\bar{F}_k=I_{n,m}(h_k)$.
\item[\textup{(iv)}] $\norm{f_k\tilde{\otimes}_{i,j} f_k}_{\FH^{\otimes ( 2(l-i-j))}}\to 0$ and $\norm{f_k\tilde{{\otimes}}_{i,j} h_k}_{\FH^{\otimes ( 2(l-i-j))}}\to 0 $ for any $0<i+j\le l-1$,
\item[\textup{(v)}] $\norm{D F_k}^2_{\FH}$, $\norm{D \bar{F}_k}^2_{\FH}$ and $\innp{D F_k,\, {D} \bar{F}_k}_{\FH}$ converge to a constant in $L^2(\Omega)$ as $k$ tends to infinity, where $D$ is the complex Malliavin derivatives. That is to say, $\mathrm{Var}(\norm{D F_k}^2_{\FH})\to 0$, $\mathrm{Var}(\norm{D \bar{ F}_k}^2_{\FH})\to 0$ and $\mathrm{Var}(\innp{D F_k,\, {D} \bar{F}_k}_{\FH} )\to 0$ as $k$ tends to infinity.
\end{itemize}
\end{thm} The above fourth moment theorems are shown by different authors with different methods \cite{ camp, chw 17, cl2}.
S.Campese \cite{camp} uses stein's method for a general context of Markov diffusion generators. \cite{cl2} is an indirect proof using Theorem~\ref{pp2} and the known fourth moment theorem for real multiple Wiener-It\^o integrals. \cite{chw 17} is an adaption of the classical arguments by D. Nualart, G.Peccati and S. Ortiz-Latorre for the one-dimensional real-valued case in \cite{np,nourorti}. That is to say, Chen et al\cite{chw 17} show the five equivalent conditions by means of $(\mathrm{i})\Rightarrow(\mathrm{ii})\Rightarrow (\mathrm{iii})\Rightarrow (\mathrm{iv})\Rightarrow (\mathrm{v})\Rightarrow (\mathrm{i})$.

\begin{ques}\it 

A typical application of Theorem~\ref{equilent cond} is to  estimate the parameter of complex OU processes. In detials, 
to model the Chandler wobble, or variation of latitude concerning with  the rotation of the earth, M. Arat\'{o}, A.N. Kolmogorov and Ya.G. Sinai \cite{arta3} 
proposed the following   stochastic  linear  equation
\begin{equation}\label{cp}
  \dif Z_t=-\gamma Z_t\dif t+ \sqrt{a}\dif \zeta_t\,,\quad t\ge 0\,,
\end{equation}%
where $Z_t=X_1(t)+\mi X_2(t)$ is a complex-valued process, $\gamma=\lambda-\mi \omega,\,\lambda>0$, $a>0$ and $\zeta_t=\frac{B^1_t+i B^2_t}{\sqrt 2} $ is a complex fractional Brownian motion
 with $(B^1_t, B^2_t)$ a two dimensional fractional Brownian  motion with a fixed Hurst index $H\in [\frac12,\frac34)$. 
The least squares estimate of the drift coefficient is given by a ratio of two Gaussian functionals:
\begin{equation}\label{fou}
  \hat{\gamma}_T=-\frac{\int_0^T \bar{Z}_t\dif Z_t}{\int_0^T \abs{Z_t}^2\dif t}=\gamma-\sqrt{a}\frac{\int_0^T \bar{Z}_t\dif \zeta_t}{\int_0^T \abs{Z_t}^2\dif t},
\end{equation}
where the integral is a divergence  integral. When $Z_0=0$, one can rewrite the above identity as
\begin{align*}
\sqrt{T} (\hat{\gamma}_T-\gamma)=\sqrt{a}\frac{I_{1,1}\big( \frac{1}{\sqrt{T}}e^{-\gamma (s-r)}\mathrm{1}_{\set{0\le r\le s\le T}} \big)}{\frac{1}{T}\int_0^T \abs{Z_t}^2\dif t}
\end{align*}
The strong consistency and asymptotic normality of the estimator $\hat{\gamma}_T$ are shown for $H\in [\frac12,\,\frac34)$ in \cite{chw 17} by means of Theorem~\ref{equilent cond}.

For the 1-dimensional real OU process, the asymptotic normality of the least squares estimator holds for $H\in (0,\frac{3}{4}]$ and a noncentral limit theorem holds for $H\in (\frac{3}{4}, 1)$ (see \cite{hn,hnz}). 
The question is whether the asymptotic normality of the estimator $\hat{\gamma}_T$ in (\ref{fou}) holds for $H\in (0,\frac{1}{2})$ or not?

\end{ques}

The following proposition is cited from \cite{chw 17}.
\begin{prop}\label{etaxinu}
Suppose that $l=m+n$ and
\begin{align}
 \eta_r&=\sum_{i+j=r}i{m\choose i}^2{n\choose j}^2i!j!f\tilde{\otimes}_{i,j}h,\label{etak}\\
 \xi_r& =\sum_{i+j=r}j {m\choose i}^2{n\choose j}^2i!j!h\tilde{\otimes}_{j, i}f,\label{xik}\\
 \nu_r&=\sum_{i+j=r} i {m\choose i}{n \choose i}{n\choose j}{m\choose j}i!j!f\tilde{\otimes}_{i, j}f,\label{nuk}
\end{align} then one has that
\begin{align}
   \mathrm{Var}(\norm{D I_{m,n}(f)}_{\FH}^2)&= \sum_{r=1}^{l-1} [(l-r)!]^2 \norm{\eta_r}^2_{\FH^{\otimes (2(l-r) )}},\label{var1}\\
   \mathrm{Var}(\norm{\bar{D} I_{m,n}(f)}_{\FH}^2)&= \sum_{r=1}^{l-1} [(l-r)!]^2 \norm{\xi_r}^2_{\FH^{\otimes (2(l-r) )}},\label{var2}\\
   \mathrm{Var}(\innp{D I_{m,n}(f),\, {D} \overline{I_{m,n}(f)}}_{\FH} )&=\sum_{r=1}^{l-1} (2m-r)! (2n-r)!  \norm{\nu_r}^2_{\FH^{\otimes (2(l-r) )}}\label{var3}.
\end{align}
\end{prop}

\begin{ques}\it
 For real multiple Wiener-It\^{o} integrals,  there is a useful inequality in p96 of \cite{np} which plays a key role in an elegant proof of the fourth moment theorem, i.e., for $G=\mathcal{I}_q(g), q\ge 2$,
 \begin{equation*}
   \mathrm{Var}\big(\frac{1}{q}\|DG\|^2_{\FH_{\mathbb{R}}}\big)\le\frac{q-1}{3q}\big(E[G^4]-3E[G^2]^2\big)\le (q-1) \mathrm{Var}\big(\frac{1}{q}\|DG\|^2_{\FH_{\mathbb{R}}}\big).
 \end{equation*}
 
From Corollary~\ref{cor 3-12} and Proposition~\ref{etaxinu}, one can easily obtain that there exists a positive constant $c(m,n)$ such that the following inequality holds.
\begin{align*}
&\quad \mathrm{Var}(\norm{D I_{m,n}(f)}_{\FH}^2)+ \mathrm{Var}(\norm{\bar{D} I_{m,n}(f)}_{\FH}^2)+ \mathrm{Var}(\innp{D I_{m,n}(f),\, {D} \overline{I_{m,n}(f)}}_{\FH} )\\
&\le c(m,n) \Big( \E[\abs{F}^4]-2\big(\E[\abs{F}^2]\big)^2 -\abs{E[{F}^2]}^2\Big).
\end{align*}
Although the reversed inequality is implied by  Theorem~\ref{equilent cond},  the question is whether one can find a direct proof of the reversed inequality. 

Developing a systematic approach to compute the moments of complex multiple Wiener-It\^o integrals is  helpful to  explore the non-Gaussian center limit theorem.
\end{ques}

\section{Appedix}

An alternative definition of  complex Hermite polynomials is as follows. 
\begin{defn}\it 
Complex Hermite polynomials are  the sequence on $\Cnum$ for $m,n\in \Nnum$,
\begin{align}
  J_{m,n}(z,\rho)&={\rho}^{m+n}(\partial^*)^m(\bar{\partial}^*)^n 1.\label{itldefn}
\end{align} where $z=x+\mi y$ and
  \begin{equation*}
    (\partial^*\phi)(z)=-\frac{\partial}{\partial \bar{z}}\phi(z)+\frac{z }{\rho}\phi(z),\quad (\bar{\partial}^*\phi)(z)=-\frac{\partial}{\partial {z}}\phi(z)+\frac{\bar{z}}{\rho }\phi(z).
  \end{equation*} are  the adjoint of the Wirtinger derivative $\frac{\partial}{\partial {z}},\,\frac{\partial}{\partial \bar{z}} $(or say: the complex creation operator) in the Hilbert space $L^2(\mu)$  with $\dif\mu = \frac{1}{\pi \rho}\exp\set{-\frac{x^2+y^2}{\rho}}\dif x\dif y$.
\end{defn}
Complex Hermite polynomials $ J_{m,n}(z,\rho)$ satisfy that \cite{cl, Ghan, is, ito} :
 \begin{itemize}
    \item[\textup{1)}]       partial derivatives:  \begin{align}
     \frac{\partial}{\partial z} J_{m,n}(z,\rho)&=m J_{m-1,n}(z,\rho),\qquad
      \frac{\partial}{\partial \bar{z}} J_{m,n}(z,\rho)=n J_{m,n-1}(z,\rho),\label{partizz}\\
      \frac{\partial}{\partial \rho} J_{m,n}(z,\rho)&=-mn J_{m-1,n-1}(z,\rho).\label{parti rho}
  \end{align}
    \item[\textup{2)}] recursion formula:
    \begin{align}
       J_{ m+1,n}(z,\rho)= {z}J_{m,n}(z,\rho)-n\rho J_{m,n-1}(z,\rho),\label{jm1n}\\
       J_{m,n+1}(z,\rho)=\bar{z}J_{m,n}(z,\rho)-m\rho J_{m-1,n}(z,\rho). \label{jmn1}
    \end{align}
\item[\textup{3)}]  Rodrigues's formula:
\begin{align}
J_{m,n}(z,\rho)&=(-\rho)^{m+n}e^{\abs{z}^2/\rho}\bar{\partial}^{m}\partial^{n}(e^{-\abs{z}^2/\rho}).
\end{align}
\item[\textup{4)}] Orthonormal basis: $\set{(m!n!\rho^{m+n})^{-\frac12}J_{m,n}(z,\,\rho)}$ is an orthonormal basis of $L_{\Cnum}^2(\mu)$.
\item[\textup{5)}] Eigenfunctions: Let $c\in \Rnum$,
\begin{equation}\label{eigen}
  [(1+\mi c)z \frac{\partial}{\partial z}+(1-\mi c)\bar{z} \frac{\partial}{\partial \bar{z}}-2\rho \frac{\partial^2}{\partial z\partial \bar{z}} ]J_{m,n}(z,\,\rho)=[m+n+\mi (m-n)c]J_{m,n}(z,\,\rho).
 \end{equation}
 \item[\textup{6)}]Monomials: 
 \begin{align}
 z^m\bar{z}^n=\sum_{r=0}^{m\wedge n}{m \choose r}{n \choose r}r!\rho^r J_{m-r,n-r}(z,\rho).
 \end{align}
  \item[\textup{7)}]Product: 
 \begin{align}
 J_{m,n}(z,\rho)J_{p,q}(z,\rho)=\sum_{i=0}^{m\wedge q}\sum_{j=0}^{n\wedge p}{m \choose i}{n \choose j}{p \choose j}{q \choose i}i!j!\rho^{i+j}J_{m+p-i-j,n+q-i-j}(z,\rho).
 \end{align}
 \end{itemize}
 
 Moreover, the following property is used to show Proposition~\ref{prop 3.5}.
\begin{prop}\label{2dim2}
Let $z=x + \mi y$ with $x,y\in \Rnum$. Then real and complex Hermite polynomials satisfy that
\begin{equation*} 
    \begin{array}{ll}
      J_{m,l-m}(z) = \sum_{k=0}^{l}{\mi^{l-k}}\sum_{r+s=k}{m \choose r}{l-m \choose s}(-1)^{l-m-s} H_k(x)H_{l-k}(y),\\
        H_k(x)H_{l-k}(y) = \frac{\mi^{l-k}}{2^{l}}\sum_{m=0}^l \sum_{r+s=m}{k \choose r}{l-k \choose s}(-1)^{s} J_{m,l-m}(z).
    \end{array}
 \end{equation*}
 Thus, both of the class $\set{ J_{k,l}(z):\,k+l=n}$ and the class $\set{ H_k(x)H_{l}(y):\,k+l=n }$ can generate the same linear subspace of $L^2(\Cnum,\nu)$  with $\dif\nu=\frac{1}{2\pi}e^{-\frac{x^2+y^2}{2}}\dif x\dif y$. 
 In addition, one has that
 \begin{equation*}
   H_l(x)H_{n-l}(y)=\sum_{k=0}^n \tensor{M}^{-1}_{l,k}H_n(x\cos\theta_k+y\sin\theta_k),\quad l=0,\dots,n,
\end{equation*}
where
for any $0<\theta_{n}<\dots<\theta_0<\pi$, the $(n+1)\times (n+1)$ matrix $(\tensor{M}_{i,j})$ is given by
\begin{align}\label{matr}
   \tensor{M}&=\tensor{M}(\theta_{0},\,\dots,\,\theta_n)\nonumber \\
     &=\left[
\begin{array}{lllll}
     (\sin\theta_0)^n & \,\,{n \choose 1}(\sin{\theta_0})^{n-1}\cos\theta_0   & \dots  &\,\, {n \choose n-1}\sin{\theta_0}(\cos\theta_0)^{n-1}&\,\, (\cos\theta_0)^{n}\\
     (\sin\theta_1)^n &\,\, {n \choose 1}(\sin{\theta_1})^{n-1}\cos\theta_1   & \dots  &\,\, {n \choose n-1}\sin{\theta_1}(\cos\theta_1)^{n-1}&\,\, (\cos\theta_1)^{n}\\
         \hdotsfor{5}\\
     (\sin\theta_n)^n &\,\, {n \choose 1}(\sin{\theta_n})^{n-1}\cos\theta_n   & \dots  &\,\, {n \choose n-1}\sin{\theta_n}(\cos\theta_n)^{n-1}&\,\, (\cos\theta_n)^{n}
\end{array}
\right].
\end{align}
\end{prop}

\section*{Acknowledgement}
 Y. Chen is supported by NSFC (No.11871079). Y. Liu is supported by NSFC (No. 11731009) and Center for  Statistical Science, PKU.

\end{document}